\documentclass[10pt]{article}%

\usepackage{marvosym}
\usepackage{booktabs}
\usepackage{multirow}
\usepackage{diagbox}
\usepackage{colortbl}

\usepackage{csquotes}

\usepackage{caption}
\usepackage{subcaption}

\usepackage[show]{ed}
\usepackage{tikz}
\usepackage{pgf}
\usetikzlibrary{arrows}
\usetikzlibrary{calc}
\usetikzlibrary{decorations.pathmorphing}
\usepackage{hyperref}
\definecolor{myblue}{rgb}{0,0,0.6}
\definecolor{lightGray}{RGB}{198,198,198}
\hypersetup{colorlinks=true, linkcolor=myblue,citecolor=myblue,filecolor=myblue,urlcolor=myblue}
\usepackage{soul}
\setstcolor{red}
\setulcolor{red}

\usepackage{rotating} % needed for \begin{sidewaysfigure}\begin{sidewaystable}

\usepackage{mathrsfs}
\DeclareMathAlphabet{\mathpzc}{OT1}{pzc}{m}{it}

\usepackage{color}
\usepackage{amsmath}
\usepackage{graphicx}
\usepackage{amsfonts}
\usepackage{amssymb}%
\usepackage{latexsym}
\usepackage{psfrag}
\usepackage{accents}

\newtheorem{theorem}{Theorem}[section]
\newtheorem{corollary}[theorem]{Corollary}
\newtheorem{definition}[theorem]{Definition}
\newenvironment{proof}[1][Proof]{\noindent \emph{#1.} }
{\hfill \ \rule{0.5em}{0.5em}}
\newtheorem{lemma}[theorem]{Lemma}
\newtheorem{proposition}[theorem]{Proposition}

\newtheorem{assumption}[theorem]{Assumption}

\numberwithin{equation}{section}
\numberwithin{table}{section}
\numberwithin{figure}{section}

\usepackage[a4paper]{geometry}
\newtheorem{remark}[theorem]{Remark}
\newtheorem{example}[theorem]{Example}

\newcommand{\nc}{\newcommand}

\makeatletter
\newcommand\RedeclareMathOperator{%
  \@ifstar{\def\rmo@s{m}\rmo@redeclare}{\def\rmo@s{o}\rmo@redeclare}%
}
\newcommand\rmo@redeclare[2]{%
  \begingroup \escapechar\m@ne\xdef\@gtempa{{\string#1}}\endgroup
  \expandafter\@ifundefined\@gtempa
     {\@latex@error{\noexpand#1undefined}\@ehc}%
     \relax
  \expandafter\rmo@declmathop\rmo@s{#1}{#2}}
\newcommand\rmo@declmathop[3]{%
  \DeclareRobustCommand{#2}{\qopname\newmcodes@#1{#3}}%
}
\@onlypreamble\RedeclareMathOperator
\makeatother

\nc{\bfx}{\mathbf{x}} % coordinate
\nc{\bfy}{\mathbf{y}} % coordinate
\nc{\bfz}{\mathbf{z}} % coordinate 
\nc{\bfu}{\mathbf{u}} % finite element vector
\nc{\bfv}{\mathbf{v}} % finite element vector
\nc{\bfw}{\mathbf{w}} % finite element vector
\nc{\bft}{\mathbf{t}} % finite element vector
\nc{\bfb}{\mathbf{b}} % right-hand side
\nc{\bfn}{\mathbf{n}} % normal vector
\nc{\bfr}{\mathbf{r}} % residual vector
\nc{\bfc}{\mathbf{c}} % residual vector

\nc{\bfA}{\mathbf{A}} % linear system
\nc{\bfB}{\mathbf{B}} % linear system
\nc{\bfC}{\mathbf{C}} % linear system
\nc{\bfR}{\mathbf{R}} % restriction matrix
\nc{\bfD}{\mathbf{D}} % algebraic partition of unity
\nc{\bfI}{\mathbf{I}} % Identity matrix
\nc{\bfM}{\mathbf{M}} % Mass matrix
\nc{\bfK}{\mathbf{K}} % Rigidity matrix
\nc{\bfP}{\mathbf{P}} % Preconditioning matrix
\nc{\bfU}{\mathbf{U}} % decomposed vector
\nc{\bfZ}{\mathbf{Z}} % decomposed vector
\nc{\bfV}{\mathbf{V}} % weakly singular operator
\nc{\bfE}{\mathbf{E}} % 
\nc{\bfH}{\mathbf{H}} % 
\nc{\bfX}{\mathbf{X}} % 
\nc{\bfId}{\mathbf{I}_{\mathrm{d}}} % Identity matrix

\nc{\bbN}{\mathbb{N}} % integers
\nc{\bbR}{\mathbb{R}} % reals
\nc{\bbP}{\mathbb{P}} % polynomial
\nc{\bbC}{\mathbb{C}} % complex
\nc{\wH}{\widetilde{H}}
\nc{\calS}{\mathcal{S}} %
\nc{\calD}{\mathcal{D}} %
\nc{\calN}{\mathcal{N}} %
\nc{\calT}{\mathcal{T}} %
\nc{\calR}{\mathcal{R}} %
\nc{\calP}{\mathcal{P}} %
\nc{\calV}{\mathcal{V}} %
\nc{\calW}{\mathcal{W}} %
\nc{\calJ}{\mathcal{J}} %
\nc{\calE}{\mathcal{E}} %
\nc{\calA}{\mathcal{A}} %
\nc{\calU}{\mathcal{U}} %
\nc{\calK}{\mathcal{K}} %
\nc{\calL}{\mathcal{L}} % 

\nc{\rouge}{\color{red}}
\nc{\bleu}{\color{blue}}
\nc{\cyan}{\color{cyan}}
\nc{\noir}{\color{black}\rm}

\nc\dif{\mathop{}\!\mathrm{d}}   % integral measure
\newcommand{\vertiii}[1]{{\left\vert\kern-0.25ex\left\vert\kern-0.25ex\left\vert #1 
    \right\vert\kern-0.25ex\right\vert\kern-0.25ex\right\vert}} % norm with triple bars
\RedeclareMathOperator{\Im}{Im} % imaginary part
\usepackage[show]{ed}
\newcommand{\cA}{{\cal A}}

\newcommand{\cG}{{\cal G}}

\newcommand{\cS}{{\cal S}}

\newcommand{\cD}{{\cal D}}

\newcommand{\bx}{x}%\mathbf{x}}
\newcommand{\by}{y}%\mathbf{y}}
\newcommand{\ba}{\widehat{a}}%\mathbf{a}}

\newcommand{\bze}{0}%\mathbf{0}}

\newcommand{\re}{{\rm e}}
\newcommand{\ri}{{\rm i}}
\newcommand{\rd}{{\rm d}}

\newcommand{\beq}{\begin{equation}}
\newcommand{\eeq}{\end{equation}}
\newcommand{\beqs}{\begin{equation*}}
\newcommand{\eeqs}{\end{equation*}}
\newcommand{\bit}{\begin{itemize}}
\newcommand{\eit}{\end{itemize}}
\newcommand{\ben}{\begin{enumerate}}
\newcommand{\een}{\end{enumerate}}
\newcommand{\bal}{\begin{align}}
\newcommand{\eal}{\end{align}}
\newcommand{\bals}{\begin{align*}}
\newcommand{\eals}{\end{align*}}
\newcommand{\bse}{\begin{subequations}}
\newcommand{\ese}{\end{subequations}}
\newcommand{\bpr}{\begin{proposition}}
\newcommand{\epr}{\end{proposition}}
\newcommand{\bre}{\begin{remark}}
\newcommand{\ere}{\end{remark}}
\newcommand{\bpf}{\begin{proof}}
\newcommand{\epf}{\end{proof}}
\newcommand{\ble}{\begin{lemma}}
\newcommand{\ele}{\end{lemma}}
\newcommand{\bco}{\begin{corollary}}
\newcommand{\eco}{\end{corollary}}
\newcommand{\bex}{\begin{example}}
\newcommand{\eex}{\end{example}}
\newcommand{\bth}{\begin{theorem}}
\newcommand{\enth}{\end{theorem}}

\newcommand{\Rea}{\mathbb{R}}
\newcommand{\Com}{\mathbb{C}}

\newcommand{\Oi}{{\Omega^-}}

\newcommand{\Oe}{{\Omega^+}}

\newcommand{\eps}{\varepsilon}

\newcommand{\half}{\frac{1}{2}}

\newcommand{\LtG}{{L^2(\bound)}}

\newcommand{\LtGt}{{\LtG\rightarrow \LtG}}

\newcommand{\HoG}{H^1(\Gamma)}

\newcommand{\tendi}{\rightarrow \infty}
\newcommand{\tendo}{\rightarrow 0}

\newcommand{\DtN}{({\rm DtN})^+}
\def\XXint#1#2#3{{\setbox0=\hbox{$#1{#2#3}{\int}$}
     \vcenter{\hbox{$#2#3$}}\kern-.5\wd0}}

\newcommand*{\N}[1]{\left\|#1\right\|}

\allowdisplaybreaks[4]
\newcommand{\tfa}{\text{ for all }}
\newcommand{\tfor}{\text{ for }}

\newcommand{\tin}{\text{ in }}
\newcommand{\ton}{\text{ on }}

\newcommand{\tand}{\text{ and }}
\newcommand{\tst}{\text{ such that }}

\newcommand{\ItD}{({\rm ItD})^-}

\newcommand{\bound}{\Gamma}%{\partial \obstacle_+}}

\usepackage{soul}

\definecolor{jwcol}{RGB}{27, 137, 18}  %{rgb}{1,0.88,0.21} changed color for visibility (david)

\definecolor{dalcol}{rgb}{0.8,0,0}

\definecolor{jeffColor}{RGB}{102, 0, 204}

\definecolor{escol}{rgb}{0,0,0.8}
\definecolor{estcol}{rgb}{0,0.5,0}
\definecolor{esnewcol}{rgb}{0,0.5,0}

\newcommand{\Cppw}{{C_{\rm ppw}}}

\newcommand{\mythmname}[1]{\emph{\bf (#1.)}}

\newcommand{\noi}{\noindent}

\newcommand{\Cqo}{C_{\rm qo}}

\newcommand{\Capprox}{C_{\rm approx}}

\newcommand{\Crel}{C_{\rm reg}}

\newcommand{\DL}{D}
\newcommand{\cDL}{\cD}
\newcommand{\operator}{\cA}
\newcommand{\altoperator}{B}
\newcommand{\Galerkinv}{v}
\newcommand{\Galerkinw}{w}

\newcommand{\pert}{K}
\newcommand{\Ai}{\operatorname{Ai}}

\title{Does the Helmholtz boundary element method suffer from the pollution effect?}
\author{
    J. Galkowski\thanks{Department of Mathematics, University College London, 25 Gordon Street, London, WC1H 0AY, UK, 
    \tt J.Galkowski@ucl.ac.uk},
\quad
    E.~A.~Spence\thanks{Department of Mathematical Sciences, University of Bath, Bath, BA2 7AY, UK, 
      \tt E.A.Spence@bath.ac.uk}}
\date{
    \today
}

\begin{document}

\maketitle

\begin{abstract}
In $d$ dimensions, accurately approximating an arbitrary function oscillating with frequency $\lesssim k$ requires $\sim k^d$ degrees of freedom.
A numerical method for solving the Helmholtz equation (with wavenumber $k$ and in $d$ dimensions) suffers from the pollution effect if,  as $k\tendi$, the total number of degrees of freedom needed to maintain accuracy grows faster than this natural threshold (i.e., faster than $k^d$ for domain-based formulations, such as finite element methods, and $k^{d-1}$ for boundary-based formulations, such as boundary element methods).

It is well known that the $h$-version of the finite element method (FEM) (where accuracy is increased by decreasing the meshwidth $h$ and keeping the polynomial degree $p$ fixed) suffers from the pollution effect, and
research over the last $\sim$ 30 years has resulted in a near-complete rigorous understanding of 
 how quickly the number of degrees of freedom must grow with $k$ to maintain accuracy (and how this depends on both $p$ and properties of the scatterer). 
 
In contrast to the $h$-FEM, at least empirically, the $h$-version of the boundary element method (BEM) does \emph{not} suffer from the pollution effect (recall that in the boundary element method
the scattering problem is reformulated as an integral equation on the boundary of the scatterer, with this integral equation then solved numerically using a finite-element-type approximation space). However, the current best results in the literature on 
 how quickly the number of degrees of freedom for the $h$-BEM must grow with $k$ to maintain accuracy fall short of proving this.

In this paper, we prove that the $h$-version of the Galerkin method applied to the standard second-kind boundary integral equations for solving the Helmholtz exterior Dirichlet problem does not suffer from the pollution effect when the obstacle is nontrapping (i.e., does not trap geometric-optic rays). 
While the proof of this result relies on information about the large-$k$ behaviour of Helmholtz solution operators, we show in an appendix %in the case
how the result can be proved using only Fourier series and asymptotics of Hankel and Bessel functions when the obstacle is a 2-d ball.
\end{abstract}

%\begin{keywords}

%\end{AMS}

\maketitle

\section{Introduction}

The boundary element method is a popular way of computing approximations to solutions of scattering problems involving the Helmholtz equation. It has long been observed, but not yet proved, that this method does not suffer from the pollution effect (in contrast to the finite element method \cite{BaSa:00}). The main result of this paper is that the $h$-version of the Helmholtz boundary element method, using the standard second-kind boundary integral equations, does not suffer from the pollution effect when the obstacle has Dirichlet boundary conditions and is smooth and nontrapping; see Theorem \ref{thm:main1} below.

In this introduction, we recap the concepts needed to understand this result, namely the Helmholtz scattering problem and the concept of nontrapping (\S\ref{sec:1.1}), a precise definition of the pollution effect (\S\ref{sec:1.2}), our current understanding of the pollution effect for finite- and boundary-element methods (\S\ref{sec:1.3}-\ref{sec:1.4}), and the definition of the boundary element-method (\S\ref{sec:1.5}-\ref{sec:1.6}).
The main result is then stated in \S\ref{sec:main_result}, the ideas behind the result are discussed in \S\ref{sec:idea}, and the result is proved in \S\ref{sec:Galerkin}--\ref{sec:proof}. 
In the special case when the obstacle is a 2-d ball, 
an alternative proof of the main result using only Fourier series and asymptotics of Hankel and Bessel functions is given in \S \ref{sec:circle}.

\subsection{The Helmholtz scattering problem}\label{sec:1.1}

The Helmholtz equation 
\beq\label{eq:Helmholtz}
\Delta u + k^2 u =0
\eeq
with wavenumber $k>0$ is arguably the simplest possible model of wave propagation.
For example, if we look for solutions of the wave equation
\beq%\label{eq:wave}
\partial_t^2U
%\pdiff{^2 U}{t^2} 
- c^2 \Delta U =0 \quad \text{ in the form} \quad 
%\eeq
%in the form 
%\beq
\label{eq:time_harmonic}
U(\bx,t) =
u(\bx) \re^{\pm \ri \omega t},
\eeq
then the function $u(\bx)$ satisfies the Helmholtz equation \eqref{eq:Helmholtz} with $k = \omega/c$ (where $\omega$ is the angular frequency and $c$ is the wave speed). 

Because the Helmholtz equation is at the heart of linear wave propagation,
much effort has gone into both studying the properties of its solutions (for example their asymptotic behaviour as $k\tendi$) and designing methods for computing the solutions efficiently; for the latter, see, e.g., the recent review articles 
\cite{ChGrLaSp:12, ErGa:12, 
GaZh:19, GaZh:22}.

The main results of this paper concern the classic scattering problem of the Helmholtz equation posed in the exterior of an obstacle with Dirichlet boundary conditions. For simplicity, we 
state our results for plane-wave scattering by an obstacle with zero Dirichlet boundary conditions; 
see Remark \ref{rem:general_Dirichlet} below for how they carry over to the general Dirichlet problem.

Let $\Oi \subset \Rea^d$, $d\geq 2$ be a bounded open set -- the ``scatterer'' or ``obstacle" -- such that its open complement $\Oe :=\Rea^d \setminus \overline{\Oi }$ is connected. 
Let $\Gamma:= \partial \Oi $; our main result requires that $\Gamma$ is smooth (i.e., $C^\infty$), although the scattering problem is well-defined for Lipschitz $\Gamma$.
Let 
$H^1_{\rm loc}(\Oe)$ be the space of functions that are in $H^1(D)$ for every bounded $D\subset \Oe$.

\begin{definition}[Plane-wave sound-soft scattering problem]\label{def:scat}
Given $k>0$ and the incident plane wave $u^I(\bx):= \exp(\ri k \bx\cdot \ba)$ for $\ba\in\Rea^d$ with $|\ba|=1$, find the total field $u\in H^1_{\rm loc}(\Oe)$ satisfying 
\beqs
\Delta u +k^2 u=0\,\,\tin\,\,\Oe,\qquad
u = 0 \,\,\ton\Gamma,
\eeqs
and such that \(u^S := u - u^I\) satisfies
\beq\label{eq:src}
\partial_r u^S -\ri ku^S = o \big(r^{(1-d)/2}\big)  \text{ as }r:=|x|\rightarrow \infty, \text{ uniformly in $x/r$}.
\eeq
\end{definition}

It is well-known that the solution of the sound-soft plane-wave scattering problem exists and is unique; see, e.g., \cite[Theorem 3.13]{CoKr:83}, \cite[Theorem 2.12 and Corollary 2.13]{ChGrLaSp:12}.

The condition \eqref{eq:src} is the \emph{Sommerfeld radiation condition}, and expresses mathematically that, with the choice $\re^{-\ri \omega t}$ in \eqref{eq:time_harmonic}, the scattered wave moves away from the obstacle towards infinity; see, e.g., \cite[\S1.1.2]{Ih:98}.

The key geometric condition that governs the behaviour of Helmholtz solutions with $k$ large is that of \emph{trapping/nontrapping} (see, e.g., \cite[Epilogue \S1]{LaPh:89}). 

\begin{definition}[Nontrapping]\label{def:nontrapping} The obstacle
$\Oi\subset \Rea^d$ is \emph{nontrapping} if $\bound$ is $C^\infty$ and,
given $R$ such that $\overline{\Oi}\subset B_R(\bze)$, there exists $T(R)<\infty$ such that 
all the billiard trajectories (a.k.a.~geometric-optic rays) 
that start in $\Oe\cap B_R(\bze)$ at time zero leave $\Oe\cap B_R(\bze)$ by time $T(R)$.
\end{definition}

\begin{figure}
\begin{center}
\begin{tikzpicture}
\begin{scope}[scale=1.3]
\draw[fill=lightGray] plot [smooth cycle] coordinates { (2,1) (1/2,1)(-1/2,0)(1/2,-1)(2,-1)(1,0)};
%\draw node at (1/2,0){$\Omega^-_{nt}$};
\end{scope}
\begin{scope}[scale=1.3, xshift=150]
\draw[fill=lightGray] plot [smooth cycle] coordinates { (2,1) (1/2,1)(-1/2,0)(1/2,-1)(2,-1) (1.8,-.3) (1, -.6)(1/2,0) (1, .6)(1.8,.3)};
%\draw node at (0,0){$\Omega^-_{t}$};
\draw[->,dashed] (1,0)--(1,-.6);
\draw[->,dashed] (1,0)--(1,.6);
\end{scope}
\end{tikzpicture}
\end{center}
\caption{\label{f:trap}On the left, 
a non-trapping object, and on the right, a trapping obstacle and one of its trapped rays.}
\end{figure}
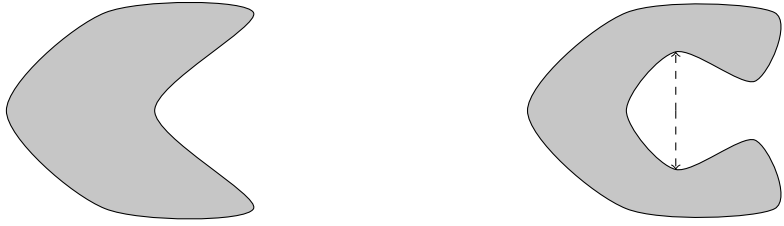

If $\Oi$ is $C^\infty$ and not nontrapping, then we say that it is \emph{trapping}; see Figure~\ref{f:trap} for an example of a non-trapping obstacle and a trapping obstacle.
The requirement that $\Gamma$ is $C^\infty$ is imposed so that when the billiard trajectories hit $\Gamma$, their reflection according to the law of geometric optics (``angle of incidence = angle of reflection'') is well-defined (see \cite{MeSj:82}). 
There has been much rigorous study of the reflection of high-frequency waves from non-smooth obstacles 
(see, e.g., \cite{Va:04, MeVaWu:13} and the references therein), 
but this does not impact the results of the present paper since we assume that $\Gamma$ is smooth (see \S\ref{sec:idea} for a discussion of why we make this assumption).

Our main results are proved under the assumption that $\Oi$ is nontrapping; in \S\ref{sec:idea} we discuss how this assumption enters our arguments.

\subsection{What is the pollution effect?}\label{sec:1.2}

\paragraph{Informal definition.}

A numerical method for solving the Helmholtz equation (with wavenumber $k$) suffers from the pollution effect if,  as $k\tendi$, \emph{the total number of degrees of freedom needed to maintain accuracy grows faster than $k^n$, where $n$ is the dimension of the physical domain in which the problem is formulated.}
Having number of degrees of freedom growing like $k^n$ is the natural threshold for the problem since  an oscillatory function with
frequency $\lesssim k$  can be accurately approximated by piecewise polynomials with $k^n$ degrees of freedom; this is expected in 1-d from the Nyquist--Shannon--Whittaker sampling theorem \cite{Wh:15, Sh:49} (see, e.g., \cite[Theorem 5.21.1]{BaNaBe:00}) and in arbitrary dimension from the Weyl law for the asymptotics of Laplace eigenvalues \cite{We:12} (see \S\ref{sec:func_calc} below for how the notion of the frequency of a function can be defined by Laplace eigenvalues).

\paragraph{Abstract framework covering both BEM and FEM.}

Let $V$ be a Hilbert space and let $\operator: V\to V'$ be a continuous, invertible linear operator, where $V'$ is the dual space of $V$. 
Given $f\in V'$, let $v\in V$ be the solution of $\operator v= f$; i.e., $v= \operator^{-1}f$. 

Let $(V_N)_{N>0}$ be an increasing sequence of finite-dimensional subspaces of $V$ 
with dimension $N$ (i.e., total number of degrees of freedom $N$), 
 such that $V_N$ are asymptotically dense in $V$, in the sense that, for all $w\in V$,  the \emph{best approximation error} $\min_{w_N \in V_N} \N{ w-w_N}_{V} \tendo$ as $N\tendi$. 

Let $v_N$ be the computed approximation in $V_N$ to $v$; we write this as $v_N= (\operator^{-1})_N f$, so that $(\operator^{-1})_N: V' \to V_N$ is the approximation of the solution operator.
 
For the finite-element method $v$ is the restriction to the computational domain of the solution $u$ of the sound-soft scattering problem (modulo any error incurred by this restriction), $V$ is the space $H^1$, and $n=d$. For the boundary-element methods we consider below, $v$ is a function on $\Gamma$ (possibly the normal derivative of $u$), $V$ is $L^2(\Gamma)$ (i.e., square integrable functions on $\Gamma$), and $n=d-1$, since the boundary $\Gamma$ is $(d-1)$-dimensional. 

\paragraph{Quasi-optimality.}

A fundamental property one seeks to prove about a sequence of approximate solutions $(v_N)_{N> 0}$ is that they are \emph{asymptotically quasi-optimal}; i.e., there exists $N_0>0$ and $\Cqo>0$ such that
\beq\label{eq:qo}
\N{v-v_N}_V \leq \Cqo \min_{w_N \in V_N} \N{ v-w_N}_{V} \quad\tfa N\geq N_0, \quad \text{where $v= \operator^{-1}f$ and $v_N =(\operator^{-1})_N f$. }
\eeq
The approximate solutions $(v_N)_{N>0}$ would be optimal if $\N{v-v_N}_V =\min_{w_N \in V_N} \N{ w-w_N}_{V}$; ``quasi-optimality'' is then optimality up to a constant factor, and  ``asymptotically" refers to the fact that \eqref{eq:qo} holds for sufficiently large $N$.

The standard analysis of finite- and boundary-element methods for the Helmholtz equation proves that, for fixed $k$, the computed solutions are asymptotically quasi-optimal (see, e.g., \cite{BrSc:08} for FEM and \cite{St:08, SaSc:11} for BEM), i.e., for each $k>0$ there exists $N_0=N_0(k)$,  depending on $k$ in some unspecified way, such that~\eqref{eq:qo} holds. 

\paragraph{Precise definition of the pollution effect.}

The pollution effect is when there exists a choice of $N$ larger than a constant multiple of $k^n$ (i.e., $N \geq \Lambda k^n$ for some $\Lambda>0$) and some choice of data ($f\in V'$) such that the smallest possible $\Cqo$ in \eqref{eq:qo} is unbounded in $k$. That is, if 
\beq\label{eq:qo_pollution}
\inf_{\Lambda>0}\limsup_{k\to \infty}\sup_{ N\geq \Lambda k^n} 
\sup_{f\in V'} \inf \bigg\{ \Cqo : \N{\operator^{-1} f-(\operator^{-1})_N f}_V \leq \Cqo \min_{w_N \in V_N} \N{ \operator^{-1} f-w_N}_{V}\bigg\}
=\infty;
\eeq
see, e.g., \cite[Definition 2.1]{BaSa:00}. Conversely, if the right-hand side of \eqref{eq:qo_pollution} is finite, then there exists $k_0, \Lambda,$ and $\Cqo$ such that for all $k\geq k_0$, $N\geq \Lambda k^{n},$ and $f\in V'$, 
\beqs
 \N{\operator^{-1} f-(\operator^{-1})_N f}_V \leq \Cqo \min_{w_N \in V_N} \N{ \operator^{-1} f-w_N}_{V};
\eeqs
i.e., $k$-uniform quasi-optimality is achieved (for all possible data) with a choice of $N$ proportional to $k^n$.

When the meshes in the FEM or BEM are \emph{quasi-uniform} (informally, all the mesh elements are of comparable size; see \cite[Definition 4.1.13]{SaSc:11} for a precise definition), then the total number of degrees of freedom $N\sim (p/h)^{n}$, where $h$ is the mesh-width and $p$ the polynomial degree.

In the $h$-version of the FEM or BEM accuracy is increased by decreasing $h$ and keeping $p$ fixed, and thus $N\sim k^n$ corresponds to $hk \sim 1$. 
For these methods, the $\sup_{ N\geq \Lambda k^n}$ in the definition of the pollution effect \eqref{eq:qo_pollution}
can then be replaced by $\sup_{ \Lambda\geq hk }$.

\subsection{The pollution effect for finite-element methods is well understood}\label{sec:1.3}

Empirically, the $h$-version of the FEM applied to the Helmholtz equation suffers from the pollution effect. Furthermore \cite{BaSa:00} proved that in two or higher dimensions the pollution effect is unavoidable for the $h$-FEM; more precisely, \cite{BaSa:00} worked in the framework of ``generalised FEMs'' introduced in \cite{BaOs:83} and proved that, in two or higher dimensions, any method with fixed polynomial degree $p$ (or, more generally, a fixed stencil) suffers from the pollution effect; see \cite[Theorem 4.6]{BaSa:00}.
    
Given that the $h$-FEM suffers from the pollution effect, two natural questions are the following.
\ben
\item[Q1] How must $h$ depend on $k$ for the quasi-optimal error estimate \eqref{eq:qo} to hold with $\Cqo$ independent of $k$?
\een
In engineering applications, the most-commonly used measure of error is the \emph{relative error} 
\beq\label{eq:rel_error}
\N{v-v_N}_V\big/ \N{v}_V.
\eeq
However, the relative error can only be small when restricting attention to a subclass of data. Indeed, since $\operator$ is assumed to be invertible, given $V_N$, we can choose $v\in V$ orthogonal to $V_N$, let $f:= \operator v$, and let $v_N := (\operator^{-1})_N f$. Then 
\beqs
\N{v-v_N}_V^2 = \N{v}^2_V + \N{v_N}^2_V \geq \N{v}^2_V, 
\eeqs
and thus the relative error cannot be small for all possible data.
\ben
\item[Q2] For a physically-relevant class of data $\widetilde{V}'\subset V'$ (such as that coming from an incident plane wave as in Definition \ref{def:scat}), how must $h$ depend on $k$ for the relative error to be controllably small? I.e., given $\eps>0$ and $\widetilde{V}'$, how must $h$ depend on $k$ and $\eps$ such that for all $f\in \widetilde{V}'$ the relative error \eqref{eq:rel_error} is $\leq \eps$?
\een

For the $h$-FEM applied to non-trapping problems, the answer to Q1 is that $h^p k^{p+1}$ must be sufficiently small, and the answer to Q2 is that $h^{2p} k^{2p+1}$ must be sufficiently small for data oscillating at scale $k^{-1}$.

These answers were first obtained for 1-d Helmholtz problems by \cite{AzKeSt:88, IhBa:95, IhBa:97} (see also \cite[Chapter 4]{Ih:98}).
Obtaining the multi-dimensional analogues of these results for a range of different FEMs remains a very active research area; see the papers
\cite{Me:95,Sa:06} (the earliest multi-dimensional results),
\cite{FeWu:09, Wu:14, ZhWu:21}
(on discontinuous Galerkin and interior penalty methods),
\cite{ChNi:18}
(on Helmholtz problems on domains with corners), 
\cite{BaChGo:17, ChNi:20, GaSpWu:20, GrSa:20}
(on variable-coefficient Helmholtz problems), and 
\cite{LiWu:19, GaChNiTo:22, GLSW1}
(on Helmholtz problems with perfectly-matched layers).
\footnote{We note that the pollution effect for Helmholtz finite element and finite difference methods can also be heuristically studied via so-called ``dispersion analysis''  \cite{HaHu:91, IhBa:95a, IhBa:95, DeArBaBo:99, Ai:04}. Here finite-element or finite-difference schemes are studied on an infinite uniform mesh for problems where an exact solution is $u(x) = \re^{\ri kx }$, and one seeks the ``discrete wavenumber'' $\widetilde{k}$ such that a numerical solution is $u_N(x_j) = \re^{\ri \widetilde{k} x_j}$, where $x_j$ are the nodes. The condition ``$h^{2p} k^{2p+1}$ sufficiently small'' (i.e., the answer to Q2) arises as the condition for $|\widetilde{k}- k|$ to be controllably small; see \cite[Theorem 3.2]{IhBa:97}, \cite[Theorem 4.22]{Ih:98}.}

There has been much research on designing FEMs that mitigate against the pollution effect; 
four directions of this research are
(i) high-order methods \cite{ZhWu:13, DuWu:15, Ch:16} and $hp$ methods  \cite{MeSa:10,MeSa:11, EsMe:12, MePaSa:13, LaSpWu:20a, LaSpWu:21},
(ii) Trefftz methods (i.e., using basis functions that are locally solutions of $\Delta u +k^2u=0$); see, e.g., the review \cite{HiMoPe:16} (in particular \cite[\S5]{HiMoPe:16},
(iii) multiscale methods involving special pre-computed test functions \cite{GaPe:15, Pe:17, BrGaPe:17, HaPe:21, FrHaPe:21}
(iv) the so-called ``discontinuous Petrov Galerkin (DPG)'' method of \cite{DeGoMuZi:12} (which is a least-squares method in a nonstandard inner product).

\subsection{The pollution effect for boundary-element methods is not yet rigorously understood}\label{sec:1.4}

The situation for the BEM is well-summarised by the following quotation from \cite{BaBeFaMaTo:17}.
\begin{quotation}
\noi It is generally admitted that Boundary Integral Equations (BIE) lead to less ``pollution effect'' than FEMs even if to our knowledge, no formal study has confirmed such a property.
\end{quotation}
Indeed, it is completely standard in the numerical-analysis and engineering communities to compute approximations to Helmholtz scattering problems via boundary integral equations using a fixed number of degrees of freedom per wavelength, i.e., $N\sim k^{d-1}$, both for Galerkin \cite{FiGaGa:04, BeVaGe:17} and collocation \cite{Ma:02, Ma:08} BEMs, and also for Nystr\"om methods \cite{BrElTu:12, LaAmGr:14, HaBaMaYo:14}. 
\footnote{Intriguingly however, \cite{Ma:16a, BaMa:17, Ma:18} recently identified a loss of accuracy similar to the pollution effect in the collocation BEM applied to interior Helmholtz problems.}

Numerical experiments indicate that, at least for obstacles without strong trapping, the $h$-BEM is quasioptimal (with constant independent of $k$) if $hk$ is sufficiently small; see \cite[\S4]{LoMe:11}, \cite[\S5]{GrLoMeSp:15}. However, 
in existing theoretical investigations \cite{BuSa:06, BaSa:07, LoMe:11, Me:12, GrLoMeSp:15, GaMuSp:19}, the best result is that the $h$-BEM is quasi-optimal (with constant independent of $k$)
for the standard second-kind BIEs for the exterior Dirichlet problem (defined in \S\ref{sec:1.5} below)
 if $hk^{4/3}$ is sufficiently small and the scatterer is smooth and convex \cite[Theorem 1.10(c)]{GaMuSp:19} (the current best results for more general domains, which are also in \cite{GaMuSp:19}, involve higher powers of $k$).
\footnote{
The only rigorous result we know of that is (i) about the convergence of a boundary integral method applied to the Helmholtz equation and (ii) valid only when $hk$ is small is that in \cite{ChRaRo:02}. Indeed, for the Helmholtz in an infinite half-plane with an impedance boundary condition solved using a collocation boundary element method and the finite-section method, \cite{ChRaRo:02} proved 
that the error is controllably small, relative to the data, if $hk$ is sufficiently small.}

The results of \cite{LoMe:11, Me:12} show, for these same BIEs, that if $\Gamma$ is analytic and the norm of the inverse of the boundary integral operator is bounded polynomially in $k$ then there exists $C_1, C_2>0$ such that the $hp$-BEM is quasi-optimal with $\Cqo$ independent of $k$ if 
\beqs
\frac{hk}{p}\leq C_1 \quad \tand\quad p\geq C_2 \log k.
\eeqs
(this is the analogous result to the $hp$-FEM results mentioned at the end of \S\ref{sec:1.3}). The abstract to \cite{LoMe:11} remarks that 
\begin{quotation}
\noi Numerical examples \dots even suggest that in many cases quasi-optimality is given under the weaker condition that $kh/p$ is sufficiently small [with $p$ fixed].
\end{quotation}
In this paper we rigorously explain this observation when the obstacle is nontrapping, showing that in this case the $h$-BEM does not suffer from the pollution effect. 

\subsection{The Helmholtz plane-wave sound-soft scattering problem solved via boundary integral equations}\label{sec:1.5}

\paragraph{The standard second-kind boundary integral equations for solving the plane-wave sound-soft scattering problem.}
In this section we recall how the solution of the plane-wave sound-soft scattering problem of Definition \ref{def:scat} can be expressed in terms of the solution of boundary integral equations involving the operators 
\beq\label{eq:DBIEs}
A_k:= \half I + \DL_k - \ri k S_k,
    \quad\tand\quad
A_k':= \half I + \DL'_k - \ri k S_k
       \eeq
where $S_k$, $\DL_k$, and $\DL'_k$ are the single-, double-, and adjoint-double-layer operators defined in \eqref{eq:SD} and \eqref{eq:D'} below. The $'$ notation is used since $A_k$ and $A_k'$ are adjoint with respect to the real-valued $L^2(\Gamma)$ inner product.

There are a variety of spaces in which one can pose equations involving $A_k$ and $A_k'$. The most natural space for solving such equations with the Galerkin method is $\LtG$ (since the inner product is local).
When $\Gamma$ is $C^1$, $S_k$, $\DL_k$, and $\DL'_k$ are compact on $\LtG$, and thus $A_k$ and $A_k'$ are compact perturbations of a multiple of the identity. Such integral operators fall into the class of ``second-kind'' operators -- see \cite[\S1.1.4]{At:97} -- and 
the solvability of integral equations involving these operators is covered by Fredholm theory.
% \footnote{Indeed, the study of analogous second-kind boundary integral equations for Laplace's equation was the main motivation for the development of this abstract functional-analysis theory; see, e.g., \cite[\S1]{Mc:00}).}.
One can then show that $A_k$ and $A_k'$ are bounded and invertible operators from $\LtG$ to itself when $\Gamma$ is smooth \cite[Theorem 3.33]{CoKr:83} (indeed, even when $\Gamma$ is only Lipschitz; see \cite[Theorem 2.7]{ChLa:07}, \cite[Theorem 2.27]{ChGrLaSp:12}).

\paragraph{How the boundary integral equations \eqref{eq:DBIEs} are obtained.}
Let $\Phi_k(x,y)$ be the fundamental solution of the Helmholtz equation 
\beqs%\label{eq:fund}
\Phi_k(x,y):= 
\frac{\ri}{4}\left(\frac{k}{2\pi |x-y|}\right)^{(d-2)/2}H_{(d-2)/2}^{(1)}\big(k|x-y|\big)=
 \left\{\begin{array}{cc}
\displaystyle{\frac{\ri}{4}H_0^{(1)}\big(k|x-y|\big)}, & d=2,\\
\displaystyle{\frac{\re^{\ri k |x-y|}}{4\pi |x-y|}}, & d=3,
\end{array}\right.
\eeqs
where $H^{(1)}_m$ denotes the Hankel function of the first kind of order $m$ (see, e.g., \cite[Equation 5.118]{Sta67}).
The single- and double-layer potentials, $\cS_k$ and $\cDL_k$ respectively, are defined for $k\in \Com$, $\phi\in L^2(\Gamma)$, and $\bx \in \Rea^d\setminus \Gamma$ by
\begin{align*}%\label{eq:SLP}
    \calS_k \phi (\bx) = \int_{\Gamma} \Phi_k (\bx,\by) \phi (\by) \dif s (\by) \quad \tand\quad
    \cDL_k \phi (\bx) = \int_{\Gamma} \dfrac{\partial \Phi_k (\bx,\by)}{\partial \nu(\by)} \phi (\by) \dif s (\by).
\end{align*}
The standard single-layer, double-layer, and adjoint-double-layer, operators are defined for $k\in \mathbb{C}$, $\phi\in \LtG$, and $x\in \Gamma$ by 
\begin{align}\label{eq:SD}
S_k \phi(\bx) := &\int_\bound \Phi_k(\bx,\by) \phi(\by)\,\rd s(\by), \qquad
\DL_k \phi(\bx) := \int_\bound \frac{\partial \Phi_k(\bx,\by)}{\partial \nu(\by)}  \phi(\by)\,\rd s(\by),\\
&\qquad \DL_k' \phi(\bx) := \int_\bound \frac{\partial \Phi_k(\bx,\by)}{\partial \nu(\bx)}  \phi(\by)\,\rd s(\by);
 \label{eq:D'}
\end{align}
when $\Gamma$ is $C^2$, the integrals defining $S_k, \DL_k,$ and $\DL'_k$ are all weakly singular; see, e.g., \cite[Page 6 and \S2.4]{CoKr:83}.
%We now state the result expressing the solution of the plane-wave sound-soft scattering problem of Definition \ref{def:scat} in terms of boundary integral equations involving $A_k$ and $A_k'$, and discuss the ideas behind its proof.

\begin{theorem}\mythmname{The plane-wave sound-soft scattering problem formulated in terms of boundary integral equations}
\label{thm:BIEs}

(i) If $u$ is solution of the plane-wave sound-soft scattering problem of Definition \ref{def:scat}, then 
\beq\label{eq:Ddirect}
A_k' \partial_{\nu} u = \partial_{\nu} u^I - \ri k u^I,\quad\text{ and }\quad u=u^I-\cS_k(\partial_\nu u).
\eeq

(ii) If $v\in\LtG$ is the solution to 
\beq\label{eq:indirect}
A_k v = - u^I,
\quad\text{ then } \quad
u= u^I + (\cDL_k -\ri k \cS_k)v
\eeq
is the solution of the plane-wave sound-soft scattering problem of Definition \ref{def:scat}.
\end{theorem}

\bpf[References for the proof and summary of the ideas]
Part (i) is proved in, e.g., \cite[Theorem 2.46]{ChGrLaSp:12}. Part (ii) is proved in, e.g., \cite[Equations 2.70-2.72]{ChGrLaSp:12}.
Both parts use that $\cD_kv$ and $\cS_kv$ satisfy the Helmholtz equation away from $\Gamma$ and satisfy the radiation condition~\eqref{eq:src}.
Part (i) uses that $u(x) = u^I(x)- \cS_k (\partial_{\nu} u)(x)$ for $x\in \Omega^+$ by Green's integral representation theorem (applied to $u^S$ in $\Oe$ and $u^I$ in $\Oi$); this is the so-called \emph{direct method}. 
Taking a linear combination of the limits of both this representation and its normal derivative as $x$ approaches $\Gamma$ from $\Oe$, 
we obtain the integral equation in \eqref{eq:Ddirect}. This idea of taking a linear combination goes back to \cite{BrWe:65, Le:65, Pa:65}, and ensures that $A_k'$ is invertible.
Part (ii) poses the ansatz that 
$u^S(x) = (\cDL_k -\ri k \cS_k)v(x)$ for 
$x\in \Omega^+$
for some unknown density $v$; this is the so-called \emph{indirect method}. 
Imposing the boundary condition that $u^S= -u^I$ on $\Gamma$, we obtain the integral equation \eqref{eq:indirect}.
\epf

\subsection{The Galerkin method and assumptions on the boundary-element space}\label{sec:1.6}

We consider solving the boundary integral equation $\operator v= f$ in $\LtG$ with the Galerkin method:~given a finite-dimensional subspace $V_N\subset \LtG$, 
\begin{align}\label{eq:Galerkin}
    \text{find }v_N \in V_N\tst\quad(\operator v_N, w_N)_{\LtG} =(f,w_N)_{\LtG}\quad\tfa w_N\in V_N.
\end{align}

The abstract framework in \S\ref{sec:1.2} involved the operator $(\operator^{-1})_N$ mapping the data to the approximate solution; 
we show in \S\ref{sec:Galerkin} below (see \eqref{eq:Galerkin_P_N}) that, for the Galerkin method, $(\operator^{-1})_N = (P_N \operator)^{-1} P_N$, where $P_N$ is the orthogonal projection from $V$ to $V_N$ and $P_N \operator$ is considered as an operator from $V_N$ to $V_N$ (after using the fact that $V$ is a Hilbert space to identify $V$ and $V'$).

The $h$-version of the boundary element method uses a sequence of approximation spaces $(V_{N_h})_{h>0}$ given by piecewise polynomials of degree $p$ for some fixed $p\geq 0$ on a sequence of meshes of diameter $h>0$; for ease of notation we let $(V_h)_{h>0}:= (V_{N_h})_{h>0}$.
It is well-known that, when the meshes are additionally \emph{shape-regular} (for each element, its width divided by the diameter of the largest inscribed ball is uniformly bounded; see \cite[Definition 4.1.12]{SaSc:11}), these subspaces satisfy the following assumption.

\begin{assumption}\label{ass:Vh}
    \((V_h)_{h> 0}\) is a sequence of finite dimensional subspaces of $\LtG$ and there exists $\Capprox>0$ such that for all    $h>0$
\beq\label{eq:bae}
        \min_{w_h \in V_h} \N{ w-w_h}_{L^2(\Gamma)} \leq \Capprox h \N{ w}_{H^1(\Gamma)} \quad\tfa w\in H^1(\Gamma).
\eeq
\end{assumption}
(Recall that $\|w\|_{H^1(\Gamma)}^2:= \N{ \nabla_{\Gamma} w }_{\LtG}^2 + \N{ w }_{L^2(\Gamma)}^2$,
where $\nabla_\Gamma$ is the surface gradient operator, defined in terms of a parametrisation of the boundary by, e.g., \cite[Equation A.14]{ChGrLaSp:12}.)

Indeed, piecewise-polynomial subspaces satisfying Assumption \ref{ass:Vh} are described in
\cite[Chapter 4]{SaSc:11}, with \cite[Theorem 4.3.22]{SaSc:11} showing that the spaces of continuous boundary-element functions 
denoted by $\cS_{\cG}^{p,0}$ \cite[Definition 4.1.36]{SaSc:11} satisfy 
Assumption \ref{ass:Vh} and 
\cite[Theorem 4.3.19]{SaSc:11} showing that the spaces of discontinuous boundary-element functions 
denoted by $\cS_{\cG}^{p,-1}$ \cite[Definition 4.1.17]{SaSc:11} satisfy 
Assumption \ref{ass:Vh}. Note that, in these cases, the constant $\Capprox$ depends on $p$.

We highlight that Assumption \ref{ass:Vh} is the only requirement on $(V_h)_{h>0}$ needed below. There are sequences $(V_h)_{h>0}$ arising from piecewise polynomials on non-quasi-uniform sequences of meshes that satisfy Assumption \ref{ass:Vh}; however, as mentioned in \S\ref{sec:1.2}, quasi-uniformity is required for the total number of degrees of freedom to $\sim (p/h)^d$.

\section{The main result:~the $h$-BEM does not suffer from the pollution effect}\label{sec:main_result}

\begin{theorem}[Quasi-optimal error estimate for $hk$ sufficiently small]\label{thm:main1}
Suppose that $\Oi$ is nontrapping
and $(V_h)_{h>0}$ satisfies Assumption \ref{ass:Vh}. 

For all $k_0>0$, there exists $\Cppw>0$ and $\Cqo>0$ such that if $\operator$ is either $A_k$ or $A_k'$,
\beq\label{eq:conditions2}
hk \leq \Cppw, \quad\tand \quad k\geq k_0,
\eeq
then, for all $f\in \LtG$, the Galerkin solution $v_N$ to \eqref{eq:Galerkin} exists, is unique, and satisfies 
\beq\label{eq:qo_main1}
\N{v-v_N}_{\LtG} \leq \Cqo \min_{w_N \in V_h} \N{v-w_N}_{\LtG}.
\eeq
\end{theorem}

The order of the quantifiers in Theorem \ref{thm:main1} (and also later results in the paper) dictates what the constants depend on; e.g., in Theorem \ref{thm:main1}, $\Cppw$ and $\Cqo$ depend on $\Omega_-$, 
the spaces $(V_h)_{h>0}$, and $k_0$, but are independent of $k$, $h$, and the choice of $A_k$ or $A_k'$.

The subscript ``ppw'' on $\Cppw$ indicates that, via \eqref{eq:conditions2}, this constant controls the number of points per wavelength. 
If the spaces $(V_h)_{h>0}$ are quasi-uniform, then $N \sim h^{-d}$, and thus Theorem \ref{thm:main1} shows that the Galerkin method is quasi-optimal (with constant independent of $k$) when the total number of degrees of freedom is a multiple of $k^d$; i.e., \emph{the $h$-BEM does not suffer from the pollution effect}.

Theorem \ref{thm:main1} covers the Galerkin method applied to $\operator v =f$ for general $f\in \LtG$. We now restrict to the case when the data comes from the plane-wave sound-soft scattering problem (i.e., the right-hand side $f$ is as described in Theorem \ref{thm:BIEs}), and bound the relative error. 
To do this, we use in the bound \eqref{eq:qo_main1} the bound \eqref{eq:bae} from Assumption \ref{ass:Vh} and the following lemma (proved in \cite[Lemma 1.3]{GaMaSp:22}), 
describing the oscillatory character of the solution $v$ in this case. 

\ble\mythmname{Bound on the unknown $v$ in the BIEs for the sound-soft scattering problem}\label{lem:Crel}
Given $k_0>0$ there exists $\Crel>0$ (with the subscript ``{\rm reg}'' standing for ``regularity'') such that 
if $\operator$ is one of $A_k, A_k'$ and $v$ is the solution to $\operator v = f$ where the right-hand side $f$ is as described in Theorem \ref{thm:BIEs}, then
\beqs%\label{eq:Crel}
\N{v}_{H^1(\Gamma)} \leq \Crel k \N{v}_{\LtG} \quad\tfa k\geq k_0.
\eeqs
\ele

\begin{corollary}[Bound on the relative error for $hk$ sufficiently small]\label{cor:main1}
Suppose that $\Oi$ is nontrapping and $(V_h)_{h>0}$ satisfies Assumption \ref{ass:Vh}. 
For all $k_0>0$, there exists $\Cppw>0$ and $\Cqo>0$ such that if 
$\operator$ is either $A_k$ or $A_k'$ and \eqref{eq:conditions2} holds, then for all data $f$ coming from the plane-wave sound-soft scattering problem
the Galerkin solution $v_N$ to \eqref{eq:Galerkin} exists, is unique, and satisfies 
\beq\label{eq:rel_err1}
\N{v-v_N}_{\LtG} \leq \Cqo \Crel hk \N{v}_{\LtG} .
\eeq
\end{corollary}

The bound \eqref{eq:rel_err1} shows that a prescribed relative error can be achieved with a choice of $h$ such that $hk\sim 1$. Indeed, given $\eps>0$, 
\beqs
\text{ if }\,\, hk \leq \min\left\{\eps (\Cqo \Crel)^{-1} , \Cppw\right\},
\quad\text{ then } \quad\N{v-v_N}_{\LtG}/\N{v}_{\LtG}\leq \eps.
\eeqs

\bre[General Dirichlet boundary conditions]\label{rem:general_Dirichlet}
The general exterior Dirichlet problem is:~given $k>0$ and $g_D \in H^{1/2}(\Gamma)$, find $u^S \in H^1_{\rm loc}(\Oe)$ such that $\Delta u^S +k^2 u^S=0$ in $\Oe$, $u^S= g_D$ on $\Gamma$, and $u^S$ satisfies the radiation condition \eqref{eq:src}.

For the indirect method, we pose the ansatz 
$u^S(x) = (\cDL_k -\ri k \cS_k)v(x)$ for 
$x\in \Omega^+$
and take the limit of this as $x$ approaches $\Gamma$ from $\Oe$ to obtain the equation $A_k v = g_D.$
Since $g_D \in H^{1/2}(\Gamma)$, this is a priori an equation in $H^{1/2}(\Gamma)$; however, since $A_k$ is bounded and invertible as an operator from $H^s(\Gamma)$ to itself for $0\leq s\leq 1$ \cite[Theorem 2.27]{ChGrLaSp:12}, and $H^{1/2}(\Gamma) \subset \LtG$, we can consider this equation in $\LtG$, and solve it using the Galerkin method as in \S\ref{sec:1.6}. 
In contrast, the exterior Dirichlet problem can only be solved by the direct method with the integral equation posed in $\LtG$ when $g_D \in H^1(\Gamma)$; see~\cite[Section 2.6]{ChGrLaSp:12}. 
\ere

\section{Discussions of the ideas behind the proof of Theorem \ref{thm:main1}}\label{sec:idea} 

The proof of Theorem \ref{thm:main1} consists of three ingredients.
\ben
\item A slight modification of a standard condition for quasi-optimality of the Galerkin method applied to operators that are a perturbation of the identity (see \eqref{eq:absG1alt} in Theorem \ref{thm:abstract} below), with this condition based on writing the Galerkin method as a projection method 
and using the result that if $\|T\|<1$ then  $I+T$ is invertible with $\|(I+T)^{-1}\|\leq (1- \|T\|)^{-1}$.
\item Bounds on the components of the boundary integral operators $S_k, \DL_k,$ and $\DL_k$ that have frequencies $> k$ (see Theorem \ref{thm:HFSD}) where the statement that a function has ``frequencies $> k$'' is understood by expanding the function in terms of eigenfunctions of the surface Laplacian on $\Gamma$ (see \S\ref{sec:func_calc}).
\een
We see in \S\ref{sec:proof} that these two ingredients prove the following result. 

\ble\label{lem:intermediate} 
Suppose $(V_h)_{h>0}$ satisfies Assumption \ref{ass:Vh}. 
For all $k_0>0$, there exists $C_1>0$ such that if $k\geq k_0$, $\operator$ is either $A_k$ or $A_k'$, and 
\beq\label{eq:idea1}
hk\big(1+\N{\operator^{-1}}_{\LtGt}\big)\leq C_1
\eeq
then, for all $f\in \LtG$, the Galerkin solution $v_N$ to \eqref{eq:Galerkin} exists, is unique, and satisfies 
\beq\label{eq:idea2}
\N{v-v_N}_{\LtG} \leq 2 \N{\operator^{-1}}_{\LtGt} \min_{w_N \in V_N}\N{v-w_N}_{\LtG}.
\eeq
\ele

The result of Theorem \ref{thm:main1} then follows from the third ingredient
(note that this is the only place where our arguments use the nontrapping assumption).
\ben
\item[3.] If $\Oi$ is non-trapping then, given $k_0>0$, there exists $C>0$ such that 
\beq\label{eq:inverse_nontrapping}
\|\operator^{-1}\|_{\LtGt}\leq C\quad \tfa k\geq k_0.
\eeq
\een

\paragraph{Discussion of Point 1.}
It is perhaps surprising that the simple condition from Theorem~\ref{thm:abstract}, combined with Points 2 and 3, gives a better result for the Galerkin method applied to $\operator$ (at least when $\Oi$ is nontrapping) than more-sophisticated conditions for quasi-optimality 
used in \cite{BuSa:06, BaSa:07, LoMe:11}, % GaMaSp:22},
which are all ultimately based on the ideas in the ``Schatz argument'' in the finite-element setting; see \cite{Sc:74, Sa:06}.

\paragraph{Discussion of Point 2.}
%Regarding 2:
The bounds on the high-frequency components of $S_k, \DL_k,$ and $\DL_k'$ in Theorem \ref{thm:HFSD} come from viewing these boundary integral operators as \emph{semiclassical pseudodifferential operators}. We do not need any of the details of these operators in this paper, but it is instructive to discuss briefly here how, on the one hand, using pseudodifferential operators to study boundary integral equations is completely standard, but, on the other hand, the full potential of these operators for studying Helmholtz problems with large $k$ has not been fully exploited.

Recall that the theory of standard pseudodifferential operators on a smooth surface $\Gamma$ can be viewed as a generalisation of Fourier analysis on the circle. 
The use of pseudodifferential properties in both the analysis and numerical analysis of boundary integral equations is both well established and current, see, e.g., the books 
\cite{SaVa:02, 
%LeMi:04, AnDa:07, 
HsWe:08,
GwSt:18}.

A class of pseudodifferential operators exists that is tailor-made for studying problems where oscillations happen at a large frequency $k$; these are precisely \emph{semiclassical pseudodifferential operators} \cite{Zw:12}, \cite[Appendix E]{DyZw:19}. The adjective ``semiclassical'' essentially means ``high frequency'', and comes from the origin of this theory in the study of how classical dynamics arise from quantum mechanics in the high-energy limit (see, e.g., \cite[\S1]{Zw:12}).

Whereas $S_k$, $D_k$, and $D_k'$ are standard pseudodifferential operators (of order $-1$; see, e.g., \cite[\S9.2.2]{HsWe:08}, \cite[\S7]{SaVa:02}, \cite[Chapter 7, Section 11]{Ta:96}), they are \emph{not} semiclassical pseudodifferential operators. Instead, each is the sum of a semiclassical pseudodifferential operator and an operator acting only on frequencies $\leq k$ that transports mass between points on the boundary connected by rays;
this decomposition was recently established in \cite[Chapter 4]{Ga:19}, with \cite[Lemma 4.27]{Ga:19} explicitly writing out the decomposition when $\Gamma$ is curved. The estimates on boundary layer operators at high frequency in Theorem \ref{thm:HFSD} were then proved using the ideas from \cite[Chapter 4]{Ga:19} in \cite[Theorem 4.3]{GaMaSp:21N}.

Finally, we note that the assumption in \S\ref{sec:1.1} that $\Gamma$ is smooth is because the theory of pseudodifferential operators is simplest on smooth domains. 
In principle, Lemma \ref{lem:intermediate} holds when $\Gamma$ is $C^M$ for some $M>0$, and one could go through the arguments to determine a sufficiently-large value of $M$; alternatively one could use more sophisticated pseudodifferential techniques to lower the regularity further; see, e.g., \cite[Chapter 13]{Ta:97}.

\paragraph{Discussion of Point 3.}
 The estimate~\eqref{eq:inverse_nontrapping} is proved in  \cite[Theorem 1.13]{BaSpWu:16}
using the following decompositions of $A_k^{-1}$ and $(A_k')^{-1}$ \cite[Theorem 2.33]{ChGrLaSp:12},
\beq\label{eq:fav_formula}
A_k^{-1}= I - \ItD \big[\DtN-\ri k\big] \quad\tand \quad (A_k')^{-1}= I -   \big[\DtN-\ri k\big] \ItD.
\eeq
 Here, $\DtN$ is the Dirichlet-to-Neumann map for the Helmholtz equation $\Delta u^S+k^2u^S=0$ in $\Oe$ satisfying the Sommerfeld radiation condition \eqref{eq:src} and $\ItD$ is the map $g\mapsto u|_{\Gamma}$ where, given $g\in \LtG$, $u\in H^1(\Oi)$ is the solution of the \emph{interior impedance problem}
\beq\label{eq:IIP}
\Delta u+ k^2 u= 0\,\, \tin \Omega^-, \quad \partial_\nu u - \ri k u = g\,\, \ton \Gamma.
\eeq
The decompositions in \eqref{eq:fav_formula} imply that bounds on $A_k^{-1}$ and $(A_k')^{-1}$ can be obtained from $k$-explicit bounds on $\DtN$ and $\ItD$. These estimates are obtained in \cite{BaSpWu:16} for non-trapping $\Oi$ (following the proof in \cite[Theorem 4.3]{ChMo:08} of the analogous bound for $\Oi$ that are star-shaped with respect to a ball).

The presence of $\DtN$ in \eqref{eq:fav_formula} is expected since $\DtN$ is essentially the solution operator for the problem (and we are using the Galerkin method applied to $A_k$ or $A_k'$ to approximate this solution operator). The map $\ItD$ appears in \eqref{eq:fav_formula} since $A_k$ and $A_k'$ can also be used to solve the interior impedance problem; see, e.g., \cite[Theorem 2.30]{ChGrLaSp:12}.

We highlight that proving $k$-explicit bounds on exterior Helmholtz solution operators is a classic problem considered since the 1960's, with interest in the interior impedance problem \eqref{eq:IIP} arising more recently \emph{both} from this problem's role in determining the behaviour of $A_k$ and $A_k'$ \emph{and} because this problem is often used as a model problem in the numerical analysis of FEMs; see the literature reviews in \cite{ChSpGiSm:20}, \cite{LaSpWu:20} (for exterior problems) \cite[\S1.2]{Sp:14}, \cite[\S1.2]{BaSpWu:16} (for both exterior and interior problems) and \cite[Sections 1.1 and 1.4]{GaLaSp:21} (for interior problems).

When $\Oi$ is trapping,  $\|\operator^{-1}\|_{\LtGt}$ grows with $k$ (see \cite{ChSpGiSm:20, LaSpWu:20}). Thus, although \eqref{eq:idea1}, \eqref{eq:idea2} give a result about convergence of the $h$-BEM for $\Oi$ trapping, this result does not show that the $h$-BEM does not suffer from the pollution effect. 
%Results about the $k$-independent quasi-optimality of the $h$-BEM in the case when $\Oi$ is trapping can be found in \cite{GaMaSp:22}, although these also do not show that the $h$-BEM does not suffer from the pollution effect. 
The experiments in \cite[Figure 2]{GrLoMeSp:15} show that, at least for a certain form of mild trapping (so-called \emph{parabolic trapping}), the $h$-BEM does not suffer from the pollution effect, although proving this remains open.

\section{Formulation of the Galerkin method as a projection method and an abstract condition for quasi-optimality}\label{sec:Galerkin}

As in \S\ref{sec:1.2}, $V$ is a Hilbert space with dual $V'$, and we let $\altoperator:V \to V'$ be a continuous, invertible, linear operator. Later we restrict attention to the case when $\altoperator$ is a perturbation of the identity, i.e., $\altoperator = I+K$, and we apply these results with $\altoperator= 2 \operator$, with $\operator$ one of $A_k'$ and $A_k$ (since $A_k$ and $A_k'$ are perturbations of $\frac{1}{2}I$ \eqref{eq:DBIEs}).

Given \(f\in V'\), let $\Galerkinv$ be the solution of the variational problem
\begin{align}\label{eq:continuous_problem}
    \text{find }\Galerkinv \in V\tst\quad\langle\altoperator \Galerkinv,\Galerkinw\rangle_{V' \times V} =\langle f,\Galerkinw\rangle_{V'\times V} \quad\tfa \Galerkinw\in V,
\end{align}
i.e. $\Galerkinv=\altoperator^{-1}f$. 
Then, given $V_N\subset V$ closed, the Galerkin approximation to $\Galerkinv$ with respect to $V_N$, $\Galerkinv_N=:(\altoperator^{-1})_Nf$, is defined as the solution of the Galerkin equations
\begin{align}\label{eq:discrete_problem}
    \text{find }\Galerkinv_N \in V_N \tst\quad\langle\altoperator \Galerkinv_N,\Galerkinw_N\rangle_{V'\times V} =\langle f ,\Galerkinw_N\rangle_{V'\times V}\quad\tfa \Galerkinw_N\in V_N.
\end{align}
We now rewrite the equations~\eqref{eq:discrete_problem} using the orthogonal projection operator $P_N:V\to V_N$. Then, $(I-P_N)$ is the orthogonal projection onto the orthogonal complement of $V_{N}$ and, in particular,
\beqs%\label{eq:Ph}
\N{(I-P_N)\Galerkinw}_V = \min_{\Galerkinw_N\in V_N}\N{\Galerkinw-\Galerkinw_N}_{V}.
\eeqs
The Galerkin equations \eqref{eq:discrete_problem} are then equivalent to the operator equation 
\beq\label{eq:Galerkin_P_N}
P_N \altoperator v_N = P_N f,\qquad v_N\in V_N,
\eeq
where we have used that $V$ is a Hilbert space to identify $V$ and $V'$ when applying $P_N$ to $\altoperator $ on the left. 
If $\altoperator= I+\pert$, then, since $v_N\in V_N$, \eqref{eq:Galerkin_P_N} simplifies to 
\beq\label{eq:Galerkin_P_N2}
(I+ P_N \pert) v_N = P_N f;
\eeq
see, e.g., \cite[\S3.1.3]{At:97}, \cite[\S13.6]{Kr:14}. Despite the fact that formally~\eqref{eq:Galerkin_P_N2} is posed on $V_N$, the operator $I+P_N\pert$ as an operator on $V$ maps $V_N\to V_N$ and hence we can study the operator $(I+P_N\pert)$ as a mapping $V\to V$.

\ble[Quasi-optimality in terms of the norm of the discrete inverse]\label{lem:discrete_inverse}
If $I+ P_N\pert : V\to V$ is invertible, then 
the Galerkin solution, $v_N$, solving~\eqref{eq:discrete_problem} exists, is unique, and satisfies
\beqs%\label{eq:absG3}
\N{v-v_N}_V \leq \N{(I+P_N \pert)^{-1}}_{V \to V} \N{(I-P_N)v}_V.
\eeqs
\ele

\bpf
Since $I+P_N \pert:V\to V$ is invertible and $I+P_N\pert:V_N\to V_N$, the solution $v_N$ to \eqref{eq:Galerkin_P_N2} exists, lies in $V_N$, and is unique as an element of $V$. Then, by \eqref{eq:continuous_problem} and \eqref{eq:discrete_problem},
\beqs
(I+P_N \pert) (v-v_N) = (I+P_N \pert) v - P_N f  = v + P_N \pert v - P_N\big( (I+\pert)v\big)  = (I-P_N)v.
\eeqs
\epf

\begin{theorem}[Sufficient condition for quasi-optimality]\label{thm:abstract}
Let $\delta>0$. If $\altoperator= I+ \pert$ and
\beq\label{eq:absG1alt}
\N{ (I-P_N) \pert (I+\pert)^{-1} }_{V\to V} \leq 1-\delta,
\eeq
then the Galerkin solution $v_N$, solving \eqref{eq:Galerkin_P_N2},  exists, is unique, and satisfies 
\beq\label{eq:absG3}
\N{v-v_N}_V \leq \delta^{-1}\N{(I+\pert)^{-1}}_{V \to V} \N{(I-P_N)v}_V.
\eeq
\end{theorem}
\bpf
The basis of the proof of \eqref{eq:absG3} is Lemma~\ref{lem:discrete_inverse} and the result that if $\|T\|<1$ then  $I+T$ is invertible with $\|(I+T)^{-1}\|\leq (1- \|T\|)^{-1}$. 
Indeed, %since
\begin{align}\label{eq:right_left}
I+P_N \pert  = I+\pert  -(I-P_N)\pert  =  \Big( I-(I-P_N)\pert  (I+\pert )^{-1}\Big)(I+\pert ).
\end{align}
Therefore, if \eqref{eq:absG1alt} holds, then 
\beqs
(I+P_N \pert)^{-1} =(I+\pert )^{-1} \Big( I-(I-P_N)\pert  (I+\pert )^{-1}\Big)^{-1}.
\eeqs
Thus, by \eqref{eq:absG1alt},
$$
\|(I+P_N\pert^{-1})\|_{V\to V}\leq \delta^{-1}\|(I+\pert)^{-1}\|_{V\to V}.
$$
and the result \eqref{eq:absG3} follows from applying Lemma~\ref{lem:discrete_inverse}
\epf

\bre\label{rem:Galerkin}
An analogous result to Theorem \ref{thm:abstract} under the condition
\beq\label{eq:absG1_alt}
\N{ (I+\pert)^{-1}(I-P_N) \pert   }_{V\to V} <1,
\eeq
is stated in, e.g., \cite[Theorem 10.1]{Kr:14}, \cite[Theorem 3.1.1]{At:97}; this result was used in the $h$-BEM context in \cite{GrLoMeSp:15}, \cite[Lemma 3.3]{GaMuSp:19}.
Here we factor out $(I+\pert)$ from the right in \eqref{eq:right_left}, rather than the left, leading to \eqref{eq:absG1alt} rather than \eqref{eq:absG1_alt}.
\ere

\section{The high-frequency behaviour of the boundary integral operators $S_k$, $\DL_k$, and $\DL_k'$}\label{sec:func_calc}

\paragraph{Functions of the surface Laplacian defined via eigenfunction expansion.}
Let $\lambda_j$ be the eigenvalues of the surface Laplacian (a.k.a.~the Laplace-Beltrami operator) $-\Delta_\Gamma$, and let $\{u_{\lambda_j}\}_{j=1}^\infty$ be an orthonormal basis for $L^2(\Gamma)$ of eigenfunctions;
i.e., 
$$
(-\Delta_{\Gamma}-\lambda_j)u_{\lambda_j}=0\quad\tand\quad \N{u_{\lambda_j}}_{L^2(\Gamma)}=1;
$$
when $\Gamma$ is the unit circle, $\{u_{\lambda_j}\}_{j=1}^\infty$ can be taken to be 
$\{\frac{1}{\sqrt{2\pi}}\re^{\ri jt}\}_{j=-\infty}^\infty$; see \S\ref{sec:circle} below.

We then define functions of $-\Delta_\Gamma$ using expansions in this basis. Precisely,
for a function $f\in L^\infty(\mathbb{R})$ and $v\in L^2(\Gamma)$, 
\beq\label{eq:func_calc}
f(-\Delta_\Gamma)v:=\sum_{j=1}^\infty f(\lambda_j)( v,u_{\lambda_j})_{\LtG}u_{\lambda_j}.
\eeq
By taking norms and using orthonormality of the basis, we see that
\beq\label{eq:sup_norm_bound}
\N{f(-\Delta_\Gamma)}_{\LtGt}\leq \N{f}_{L^\infty(\mathbb{R})}.
\eeq

\paragraph{Frequency cut-offs defined as functions of the surface Laplacian.}

With $u_{\lambda_j}$ defined above, we say that ``a function $v$ has frequency $\geq M$'' if, for some $a_{\lambda_j}\in\mathbb{C}$
\beqs
v= \sum_{\lambda_j \geq M^2} a_{\lambda_j} u_{\lambda_j}.
\eeqs

For $\chi \in C_{\rm comp}^\infty(\Rea)$ with $\chi \equiv 1$ on $U\subset \mathbb{R}$, the operator $(I-\chi(-k^{-2}\Delta_\Gamma))$ therefore restricts to functions with frequencies outside the set $kU$. In particular, if $\chi \equiv 1$ in a neighbourhood of $[-1,1]$, then $(I-\chi(-k^{-2}\Delta_\Gamma))$ restricts to functions with frequencies $> k$. 

\begin{theorem}[The high-frequency behaviour of $S_k$, $D_k$, and $D_k'$]\label{thm:HFSD}
Suppose $\chi \in C_{\rm comp}^\infty(\mathbb{R})$ with $\chi \equiv 1$ in a neighborhood of $[-1,1]$. Then for all $k_0>0$ there exists $C>0$ such that for all $k\geq k_0$
\begin{equation}
\label{e:highFreq}
\begin{gathered}
\N{(I-\chi (-k^{-2}\Delta_\Gamma))\DL_k}_{L^2(\Gamma)\to H^1(\Gamma)}+\N{(I-\chi (-k^{-2}\Delta_\Gamma))\DL'_k}_{L^2(\Gamma)\to H^1(\Gamma)}\leq Ck,\\
\N{(I-\chi (-k^{-2}\Delta_\Gamma))S_k}_{L^2(\Gamma)\to H^1(\Gamma)}\leq C.
\end{gathered}
\end{equation}
\end{theorem}

By the discussion above, we see that the bounds in \eqref{e:highFreq} are bounds on the outputs of $\DL_k, \DL_k'$, and $S_k$ with frequencies $> k$.

\

\bpf[References for the proof of Theorem \ref{thm:HFSD}]
This is proved in \cite[Theorem 4.4]{GaMaSp:21N}; see also \cite[Theorem 3.1, Remark 4.2]{GaMaSp:22}; we note that the key ingredient \cite[Lemma 3.10]{GaMaSp:21N} is a simplified version of \cite[Lemma 4.27]{Ga:19}, and the semiclassical analogue of 
\cite[Chapter 7, \S11]{Ta:96} and \cite[Theorem 8.4.3]{HsWe:08}.
\epf

\begin{lemma}[Smoothing property of compactly-supported functions of $-k^{-2}\Delta_\Gamma$]
\label{l:basicElliptic}
Suppose that $f\in L^\infty_{\rm comp}(\mathbb{R})$. Then for all $s\geq 0$ there exists $C_{s, f}>0$ such that
\beq\label{eq:Sobolev_bound}
\N{f(-k^{-2}\Delta_\Gamma)}_{L^2(\Gamma)\to H^s(\Gamma)}\leq C_{s,f}k^s \quad \tfa k>0.
\eeq
\end{lemma}
\begin{proof}
By elliptic regularity, given $\ell>0$ there exists $C_\ell$ such that for all $v$
$$
\N{v}_{H^{2\ell}(\Gamma)}\leq C_{\ell}\Big(\N{(-\Delta_\Gamma)^\ell v}_{L^2(\Gamma)}+\N{v}_{L^2(\Gamma)}\Big);
$$
this follows from interior regularity for second-order elliptic operators with variable coefficients; see, e.g.,~\cite[Section 6.3.1]{Ev:98}. Thus 
\beq\label{eq:ellip_temp1}
\N{f(-k^{-2}\Delta_\Gamma)v}_{H^{2\ell}(\Gamma)}\leq C_{\ell}\Big(\|(-\Delta_\Gamma)^\ell f(-k^{-2}\Delta_\Gamma) v\|_{L^2(\Gamma)}+\N{f(-k^{-2}\Delta_\Gamma)v}_{L^2(\Gamma)}\Big).
\eeq
By \eqref{eq:sup_norm_bound}, the last term on the right-hand side of \eqref{eq:ellip_temp1} is bounded by $C\|v\|_{\LtG}$ for $C$ depending on $f$ but independent of $k$. For the first term on the right-hand side  of \eqref{eq:ellip_temp1} we use that fact that
$s^\ell f(s)\in L^\infty$ (since $f$ has compact support) to see that
\begin{align*}
\|(-\Delta_\Gamma)^\ell f(-k^{-2}\Delta_{\Gamma})\|_{\LtGt}
=k^{2\ell}\|(-k^{-2}\Delta_\Gamma)^\ell f(-k^{-2}\Delta_{\Gamma})\|_{\LtGt}
&\leq k^{2\ell}\|s^\ell f(s)\|_{L^\infty}\\
&\leq \widetilde{C}_\ell k^{2\ell},
\end{align*}
for some $\widetilde{C}_\ell>0$. Using these bounds in \eqref{eq:ellip_temp1} we obtain the bound \eqref{eq:Sobolev_bound} for even $s$. The bound for odd $s$ then follows by interpolation (see, e.g., \cite[Theorems B.2]{Mc:00}) using the fact that $H^s(\Gamma)$ is an interpolation scale (see, e.g., \cite[Theorem B.11]{Mc:00}).
\end{proof}

\section{Proof of Theorem \ref{thm:main1}}\label{sec:proof}

It is sufficient to prove Lemma \ref{lem:intermediate}, since Theorem \ref{thm:main1} then follows from the bound \eqref{eq:inverse_nontrapping}.

As described in \S\ref{sec:idea}, we use Theorems \ref{thm:abstract} and \ref{thm:HFSD}.
We apply the former with $\altoperator = 2\operator$, so that $\pert= 2\operator-I$, and $\delta=1/2$. Thus, we only need to prove that there exists $C_1>0$ 
(independent of $h$ and $k$) such that if \eqref{eq:idea1} holds then 
\beqs
\N{(I-P_N) (2\operator-I) (2\operator)^{-1}}_{\LtGt} \leq \tfrac{1}{2}.
\eeqs
By the bound \eqref{eq:bae} from Assumption \ref{ass:Vh}, it is sufficient to show that there exists $C_1>0$ 
(independent of $h$ and $k$) such that if \eqref{eq:idea1} holds then 
\beqs
h\Capprox \N{(2\operator-I) (2\operator)^{-1}}_{\LtG\to \HoG} \leq \tfrac{1}{2}.
\eeqs
We therefore only need to show that 
\beq\label{eq:flow1}
\N{(2\operator-I)(2\operator)^{-1}}_{\LtG\to\HoG} \leq C_2 k \big( 1 + \N{\operator^{-1}}_{\LtGt}\big),
\eeq
for some $C_2>0$ (independent of $h$ and $k$),
and then the result holds with $C_1:= (2 \Capprox C_2)^{-1}$.

To prove~\eqref{eq:flow1}, let $\chi \in C_{\rm comp}^\infty(\mathbb{R})$ with $\chi\equiv 1$ in a neighborhood of $[-1,1]$. 
Since $1= \chi + (1-\chi)$,
\begin{align}\nonumber
(2\operator -I)(2\operator)^{-1}&=\chi(-k^{-2}\Delta_\Gamma)(2\operator -I)(2\operator)^{-1}+\big(I-\chi(-k^{-2}\Delta_\Gamma)\big)(2\operator-I)(2\operator)^{-1}\\
&=\chi(-k^{-2}\Delta_\Gamma)(I -(2\operator)^{-1})+\big(I-\chi(-k^{-2}\Delta_\Gamma)\big)(2\operator-I)(2\operator)^{-1}.\label{eq:lateral1}
\end{align}
To deal with the first term on the right-hand side of \eqref{eq:lateral1}, we use Lemma~\ref{l:basicElliptic} applied with $f=\chi$ to find that
\begin{align}\nonumber
\N{\chi(-k^{-2}\Delta_\Gamma)(I-(2\operator)^{-1})}_{\LtG\to \HoG}&\leq \N{\chi(-k^{-2}\Delta_\Gamma)}_{\LtG\to \HoG}\big(1+\N{(2\operator)^{-1}}_{L^2\to L^2}\big)\\
&\leq C_3 k(1+\|\operator^{-1}\|_{L^2\to L^2}),\label{eq:lateral_flow1}
\end{align}
for some $C_3>0$ (independent of $h$ and $k$).
We now consider the second term on the right-hand side of \eqref{eq:lateral1} when $\operator=A_k$; the proof when $\operator=A_k'$ follows in exactly the same way, just replacing $D_k'$ by $D_k$. By the definition of $A_k$ \eqref{eq:DBIEs} and Theorem \ref{thm:HFSD},
\begin{align}\nonumber
&\N{(I-\chi(-k^{-2}\Delta_\Gamma))(2\operator-I)(2\operator)^{-1}}_{\LtG\to \HoG} \\ \nonumber
&\hspace{4cm}=\N{(I-\chi(-k^{-2}\Delta_\Gamma))(\DL_k - \ri kS_k)(2 A_k)^{-1}}_{\LtG\to \HoG} \\ \nonumber
&\hspace{4cm}\leq \N{(I-\chi(-k^{-2}\Delta_\Gamma))(\DL_k - \ri k S_k)}_{\LtG\to \HoG} \tfrac{1}{2} \N{A_k^{-1}}_{\LtGt}\\
&\hspace{4cm}\leq C_4 k \N{A_k^{-1}}_{\LtGt},\label{eq:lateral_flow2}
\end{align}
for some $C_4>0$ (independent of $h$ and $k$). Combining \eqref{eq:lateral_flow1} and \eqref{eq:lateral_flow2} we obtain \eqref{eq:flow1}, and the proof is complete.

\appendix
\section{
A simple proof of Theorem \ref{thm:main1} when $\Gamma$ is the unit circle}\label{sec:circle}

Ultimately, the most flexible tools to study the large-$k$ behaviour of Helmholtz boundary integral operators come from
 semiclassical analysis. Nevertheless, in the special case when $\Gamma$ is the unit circle, Theorem \ref{thm:main1} can be proved using only results about Fourier series and the asymptotics of Bessel and Hankel functions. The advantage of the latter proof is that it only uses classical tools of applied mathematics; furthermore, since we write this proof mirroring the general proof in \S\ref{sec:proof}, we hope it makes the ideas in \S\ref{sec:proof} clearer.

\subsection{Recap of Fourier-series results.} Suppose $\Gamma$ is the unit circle, with parametrisation $\gamma(t) = (\cos t,\sin t)$ for $t\in [0,2\pi)$. 
With this parametrisation, $L^2(\Gamma)$ is isometrically isomorphic to $L^2(0,2\pi)$. Given $v \in L^2(0,2\pi)$, define the $n$th Fourier coefficient of $v$ by
\beqs%\label{eq:Fourier_coeff}
\widehat{v}_n:= \frac{1}{\sqrt{2\pi}}\big( v , \re^{\ri n \cdot}\big)_{L^2(0,2\pi)} =\frac{1}{\sqrt{2\pi}}\int_0^{2\pi} \re^{-\ri nt} v(t) \, \rd t,
\quad
\text{ so that } 
\quad v(t) = \sum_{n=-\infty}^\infty \widehat{v}_n  \frac{\re^{\ri n t}}{\sqrt{2\pi}}
\eeqs
as an $L^2$ function.
Parseval's theorem states that 
\beq\label{eq:Parseval}
\N{v}^2_{L^2(0,1)} = \sum_{n=-\infty}^\infty |\widehat{v}_n|^2
\quad\text{ and thus }\quad
%Next, recall that the $H^1(\Gamma)$ norm is defined in terms of Fourier coefficients by
%\beq\label{eq:H1_norm}
\|v\|_{H^1(0,1)}^2:=\sum_{m=-\infty}^\infty (1+m^2)|\widehat{v}_m|^2.
\eeq
%\eeq

\subsection{Results about the eigenvalues of $2A_k$.}

When $\Gamma$ is a circle, $A_k= A_k'$ since $\DL_k=\DL'_k$; this follows from the definitions of $\DL_k$ and $\DL'_k$ and the geometric property that $(x-y)\cdot \nu(y) =(x-y)\cdot \nu(x)$ for $x,y$ on a circle.

\ble[Expression for eigenvalues of $2A_k$ in terms of Bessel and Hankel functions]
If 
\beq
\label{e:lambda}
\lambda_m(k):= \pi k H_{|m|}^{(1)}(k) \Big( \ri J_{|m|}'(k) + J_{|m|}(k)\Big).
\eeq
then 
\beq\label{eq:2A_expansion}
(2A_k v)(t) = \frac{1}{\sqrt{2\pi}}\sum_{m=-\infty}^\infty \lambda_m(k) \widehat{v}_m \re^{\ri m t}.
\eeq
\ele

\bpf[References for the proof]
See, e.g., \cite[\S4 (in particular Equation 4.4)]{Kr:85} or \cite[Lemma 4.1]{DoGrSm:07}.
\epf

\begin{theorem}[Sign property of eigenvalues of $2A_k$ on unit circle]\label{thm:DGS}

If $\Gamma$ is the unit circle, then there exists $k_0>0$ such that, for all $m$ and for all $k\geq k_0$, 
\beqs%\label{eq:coercivity}
\Re\lambda_m(k)\geq 1.
\eeqs
\end{theorem}

\bpf[Reference for the proof]
This is proved in \cite[Theorem 4.2]{DoGrSm:07} using asymptotics of Bessel and Hankel functions.
\epf

\

The only other rigorous result about the eigenvalues $\lambda_m(k)$ that we need is the following.

\ble[Asymptotics of $\lambda_m(k)$ as $m\tendi$ with $m>k$]
\label{lem:lambda_bound1}
Let $z:=k/m$. Then, for all $\delta>0$ there exists $C>0$ such that for $0<z<1-\delta$,
\beqs
|\lambda_m(k) -1|\leq Cz.
\eeqs
\ele

\bpf
We first review some standard facts about
uniform asymptotics for the Bessel functions $J_m(mz)$ and $H^{(1)}_m(mz)$ \cite{Ol:54}, \cite[Section 10.20]{Di:22}, where $m\geq 0$ and $z<1-\delta$.
We define the decreasing bijection  $(0,1) \ni z \mapsto \zeta(z) \in (0,\infty)$ by
\begin{equation*}%\label{e:zetadef}
 \zeta := 
 %\phantom{-}
 \frac 32 \bigg(\int_z^1 t^{-1} (1-t^2)^{1/2}\rd t\bigg)^{2/3},
 %-\frac 32 \bigg(\int_1^z t^{-1} (t^2-1)^{1/2}\rd t\bigg)^{2/3}
 \end{equation*}
and recall the definition of the Airy function, $\Ai$,
\begin{equation*}%\label{e:airydef}\begin{aligned}
 \Ai(x) := \frac 1 \pi \int_0^\infty \cos\left(\dfrac {t^3}3 + xt\right)\rd t.
%\end{aligned}
\end{equation*}
By \cite[Section 9.7]{Di:22}, for $|\arg (x)|<\pi-\delta$, 
\begin{equation}\label{e:airypos}\begin{gathered}
\Ai(x)=\exp\big(- \tfrac 23 x^{3/2}\big)\bigg(\frac{1}{2\sqrt{\pi}}x^{-1/4} +O(x^{-7/4})\bigg), \\
 \Ai'(x) =  \exp\big(- \tfrac 23 x^{3/2}\big)\bigg(-\frac{1}{2\sqrt{\pi}}x^{1/4}+O(x^{-5/4})\bigg), 
\end{gathered}\end{equation}
where the branch cut is taken on $x\in (-\infty, 0)$.  Moreover, by~\cite[Section 9.9]{Di:22}, $|\Ai(x)|, |\Ai'(x)|>0$ for $x\notin (-\infty,0)$. 
Then, by \cite[Section 10.20]{Di:22}, uniformly for $m\geq 1$ and $0<z<1$,
\begin{equation}\label{e:jyaibi}\begin{split}
J_m(m z) &= \left(\frac{4 \zeta} {1-z^2}\right)^{1/4} \left(m^{-1/3} \Ai(m^{2/3}\zeta)+ O\Big(m^{-5/3}\zeta^{-1/2}\Ai'(m^{2/3}\zeta)\Big) \right),\\
J'_m(m z) &= -\frac{2}{z} \left(\frac{1-z^2}{4 \zeta} \right)^{1/4} \left(m^{-2/3} \Ai'(m^{2/3}\zeta) +O\Big( m^{-4/3}  \zeta^{1/2}\Ai(m^{2/3}\zeta)\Big)\right),\\
H_{m}^{(1)}(m z) &=   2\re^{-\pi \ri/3}\left(\frac {4 \zeta }{1-z^2}\right)^{1/4} \left(m^{-1/3} \Ai(\re^{2\pi\ri/3}m^{2/3}\zeta) +O\Big(m^{-5/3}\zeta^{-1/2}\Ai'(\re^{2\pi\ri/3}m^{2/3}\zeta)\Big) \right).
\end{split}\end{equation}
%where $\langle \zeta \rangle := (1+ |\zeta|^2)^{1/2}$.
Next, note that when $0<z< 1-\delta$, there exists $c_\delta>0$ such that $\zeta \geq c_\delta$ and thus we can use the asymptotics for Airy functions~\eqref{e:airypos}. Putting these asymptotics in~\eqref{e:jyaibi} and using the definition of $\lambda_m(k)$ \eqref{e:lambda}, we obtain that for any $\delta>0$, there exists $C>0$ such that
$$
\Big| \lambda_m(k)-1\Big|=\Big|\pi k H_{|m|}^{(1)}(k) \Big( \ri J_{|m|}'(k) + J_{|m|}(k)\Big)-1\Big|\leq C\frac{k}{m},\quad \tfor m>(1+\delta)k,
$$
as claimed.
\epf

\subsection{Proof of Theorem~\ref{thm:main1} when $\Gamma$ is the unit circle}
Observe that in the case of the circle, the functional calculus for the surface Laplacian reviewed in Section~\ref{sec:func_calc} is simply the theory of Fourier multipliers; i.e. the collection $\{\frac{1}{\sqrt{2\pi}}\re^{\ri mt}\}_{m=-\infty}^\infty$ is an orthonormal basis of eigenfunctions of $-\Delta_\Gamma$ satisfying 
$$
(-\Delta_\Gamma -m^2)\dfrac{1}{\sqrt{2\pi}}\re^{\ri m t}=0,\qquad \left\|\dfrac{1}{\sqrt{2\pi}}\re^{\ri mt}\right\|_{L^2(\Gamma)}=1.
$$
Thus \eqref{eq:func_calc} becomes
\begin{equation*}
%\label{e:functionalCircle}
f(-\Delta_\Gamma)v:=\frac{1}{\sqrt{2\pi}}\sum_{m=-\infty}^\infty f(m^2) \widehat{v}_m \re^{\ri mt}.
\end{equation*}
To prove Theorem~\ref{thm:main1}, we only need to check the conditions of Theorem~\ref{thm:abstract} with $I+\pert=2\operator=2A_k$. 
Using Assumption \ref{ass:Vh} as in the beginning of \S\ref{sec:proof}, we see that we only need to prove the bound \eqref{eq:flow1}.
The expansion \eqref{eq:2A_expansion} implies that 
\beqs%\label{eq:2A_expansion}
\big((2A_k)^{-1} v\big)(t) = \frac{1}{\sqrt{2\pi}}\sum_{m=-\infty}^\infty \big(\lambda_m(k)\big)^{-1} \widehat{v}_m \re^{\ri m t}.
\eeqs
By Theorem~\ref{thm:DGS}, and the fact that $|\lambda_m|\geq |\Re \lambda_m|\geq 1$,
\beq\label{eq:DGScor}
\sup_m |\lambda_m (k)|^{-1}\leq 1.
\eeq
Therefore, by taking $L^2$ norms and using orthonormality (in a similar way to how \eqref{eq:sup_norm_bound} is obtained),  we obtain the bound \eqref{eq:inverse_nontrapping} in this setting
\beq\label{eq:inverse_circle}
\|(2 A_k)^{-1}\|_{\LtGt}\leq \sup_m |\lambda_m (k)|^{-1} 
\leq 1.
\eeq

To prove the bound \eqref{eq:flow1}, we therefore only need to show that 
\beq\label{eq:circle_desired}
\|(2A_k-I)(2A_k)^{-1}\|_{\LtG\to\HoG} \leq C k;
\eeq
we do this using the splitting \eqref{eq:lateral1} with $\chi\in C_{\rm comp}^\infty(\mathbb{R};[0,1])$ with $\chi\equiv 1$ on $[-1-\eps,1+\eps]$. 
To deal the first term on the right-hand side of \eqref{eq:flow1}, we observe that, by~\eqref{eq:inverse_circle}
\beq\label{eq:theend1}
\N{I-(2 A_k)^{-1}}_{\LtGt}\leq 2.
\eeq
The definition of the $H^1$ norm in \eqref{eq:Parseval}, along with the compact support of $\chi$ and Parseval's theorem in \eqref{eq:Parseval}, implies the following analogue of Lemma~\ref{l:basicElliptic} with $s=1$
\beq\label{eq:theend2}
\N{\chi(-k^{-2}\Delta_\Gamma)}_{L^2(\Gamma)\to H^1(\Gamma)}^2\leq \sup_{m} \Big[(1+m^2)\big|\chi(k^{-2}m^2)\big|\Big]\leq C k^2.
\eeq
Combining \eqref{eq:theend1} and \eqref{eq:theend2}, we obtain the following bound on the first term of the right-hand side of \eqref{eq:lateral1}
\begin{equation}
\label{e:lowFreqAgain}
\N{\chi(-k^{-2}\Delta_\Gamma)(I-(2A_k)^{-1})}_{L^2(\Gamma)\to H^1(\Gamma)}\leq Ck.
\end{equation}

To deal the second term on the right-hand side of \eqref{eq:lateral1}, we
observe that
$$
\big(I-\chi(-k^{-2}\Delta_\Gamma)\big)(2A_k - I)(2A_k)^{-1}\re^{\ri mt}=\big(1-\chi(-k^{-2}m^2)\big)\frac{\lambda_m(k)-1}{\lambda_m(k)}\re^{\ri mt}.
$$
Thus, using the Fourier representation of $(1-\chi(-k^{-2}\Delta_\Gamma)\big)(2A_k - I)(2A_k)^{-1}$ and the definition of the $H^1(\Gamma)$ norm \eqref{eq:Parseval}, we find that
\begin{align*}
\N{\big(I-\chi(-k^{-2}\Delta_\Gamma)\big)(2A_k - I)(2A_k)^{-1}}^2_{L^2(\Gamma)\to H^1(\Gamma)}
&\leq \sup_m\bigg[(1+m^2) \big|\big(1-\chi(k^{-2}m^2)\big)\big|\frac{|\lambda_m(k)-1|}{|\lambda_m(k)|}\bigg].
\end{align*}
By the definition of $\chi$, $(1-\chi(k^{-2}m^2))=0$ when $m^2\leq (1+\eps)k^2$, and $(1-\chi(k^{-2}m^2))\leq 1$ for all $m$; therefore 
\begin{align*}
\N{\big(I-\chi(-k^{-2}\Delta_\Gamma)\big)(2A_k - I)(2A_k)^{-1}}^2_{L^2(\Gamma)\to H^1(\Gamma)}
&\leq \sup_{m^2\geq(1+\eps)k^2} \bigg[(1+m^2) %\big|\big(1-\chi(k^{-2}m^2)\big)
\frac{|\lambda_m(k)-1|}{|\lambda_m(k)|}\bigg].
\end{align*}

Observe from \eqref{e:lambda} that $\lambda_m(k)= \lambda_{|m|}(k)$; the regime $m^2\geq (1+\eps)k^2$ is therefore exactly that covered by Lemma \ref{lem:lambda_bound1} (with $(1+\eps)^{-1/2}= 1-\delta$). Using Lemma \ref{lem:lambda_bound1}  along with
\eqref{eq:DGScor}, 
we obtain that
\begin{align}\nonumber
%\begin{aligned}
\|(I-\chi(-k^{-2}\Delta_\Gamma))(2A_k - I)(2A_k)^{-1}\|^2_{\LtG\to \HoG}&\leq \sup_{m^2\geq (1+\delta)k^2}
\bigg[(1+m^2)\frac{|\lambda_m(k)-1|}{|\lambda_m(k)|}
\bigg]\\
&\hspace{-1cm}\leq C^2 \sup_{m^2\geq (1+\delta)k^2}\bigg[(1+m^2)\frac{k^2}{m^2}\bigg]\leq C'k^2.
%\end{aligned}
\label{e:highFreqAgain}
\end{align}
Combining the bounds \eqref{e:lowFreqAgain} and \eqref{e:highFreqAgain}, we obtain \eqref{eq:circle_desired}, and the proof is complete.

\section*{Acknowledgements}
Both Francesco Andriulli
(Ecole Nationale Sup\'erieure Mines-T\'el\'ecom Atlantique)
 and Th\'eophile Chaumont-Frelet (INRIA, Nice) independently suggested to EAS to look at the particular case of the circle after EAS's talk on \cite{GaMuSp:19} at the conference IABEM 2018. EAS thanks C\'ecile Mailler (University of Bath), Pierre Marchand (INRIA, Paris), and Manas Rachh (Flatiron Institute) for subsequent useful discussions on the circle case. Both JG and EAS thank Alastair Spence (University of Bath) for useful comments on an early draft of the paper and also the anonymous referees for their careful reading of a previous version of the paper.
JG was supported by EPSRC grant EP/V001760/1, and EAS was supported by EPSRC grant EP/R005591/1.

\footnotesize{
\bibliographystyle{plain}
\bibliography{biblio_combined_sncwadditions.bib}

\begin{thebibliography}{100}
\bibitem{Wh:15}
E.~T. Whittaker.
\newblock On the functions which are represented by the expansions of the
  interpolation-theory.
\newblock {\em Proceedings of the Royal Society of Edinburgh}, 35:181--194,
  1915.
\end{thebibliography}
}

\end{document}